\documentclass{amsart}

\usepackage[all]{xy}

\theoremstyle{plain}
\newtheorem{thm}{Theorem}[section]
\newtheorem{prop}[thm]{Proposition}

\newtheorem{cor}[thm]{Corollary}

\theoremstyle{definition}
\newtheorem{dfn}[thm]{Definition}

\theoremstyle{remark}

 \newcommand{\unter}[2]{\genfrac{}{}{0pt}{}{#1}{#2}}

\begin{document}

\title[Universal deformation rings and certain algebras of dihedral type]{Universal deformation rings of modules for algebras of dihedral type of polynomial growth}

\author{Frauke M. Bleher}
\address{F.B.: Department of Mathematics\\University of Iowa\\
14 MacLean Hall\\Iowa City, IA 52242-1419, U.S.A.}
\email{frauke-bleher@uiowa.edu}
\thanks{The first author was supported in part by NSA Grant H98230-11-1-0131.}
\author{Shannon N. Talbott}
\address{S.T.: Department of Mathematics\\College of Mount St. Joseph\\5701 Delhi Road\\
Cincinnati, OH 45233-1670, U.S.A.}


\begin{abstract}
Let $k$ be an algebraically closed field, and let $\Lambda$ be an algebra of dihedral
type of polynomial growth as classified by Erdmann and Skowro\'{n}ski. We describe
all finitely generated $\Lambda$-modules $V$ whose stable endomorphism rings are
isomorphic to $k$ and determine their universal deformation rings $R(\Lambda,V)$.
We prove that only three isomorphism types occur for $R(\Lambda,V)$: $k$, $k[[t]]/(t^2)$
and $k[[t]]$.
\end{abstract}

\maketitle


\section{Introduction}
\label{s:intro}

Let $k$ be an algebraically closed field of arbitrary characteristic, and let $\Lambda$
be a finite dimensional algebra over $k$. Given a finitely generated $\Lambda$-module
$V$, it is a natural question to ask over which complete local commutative Noetherian 
$k$-algebras $R$ with residue field $k$ the module $V$ can be lifted. 
If $\Lambda$ is self-injective and the stable endomorphism ring of $V$ is isomorphic to $k$, 
then there exists a particular complete local commutative Noetherian $k$-algebra 
$R(\Lambda,V)$ with residue field $k$ and a particular lift $U(\Lambda,V)$ of $V$ over 
$R(\Lambda,V)$ which is universal with respect to all
isomorphism classes of lifts (i.e. deformations) of $V$ over all such $k$-algebras $R$ (see 
\cite{bv} and \S\ref{s:prelim}). 
The ring $R(\Lambda,V)$ is called the universal deformation ring of $V$ and
$U(\Lambda,V)$ is called the universal deformation of $V$. 
Since for every $R$ and every lift $M$ of $V$ over $R$ there exists a unique specialization 
morphism
$R(\Lambda,V)\to R$ with respect to which $U(\Lambda,V)$ specializes to $M$, 
the pair $(R(\Lambda,V),U(\Lambda,V))$ encompasses all 
lifts of $V$ over all complete local commutative Noetherian $k$-algebras with residue field 
$k$.

The question of lifting modules
has a long tradition when $\Lambda$ is equal to the group ring $kG$ of a finite group 
$G$ and $k$ has positive characteristic $p$. In this case, one not only studies lifts of $V$ to 
complete local commutative Noetherian $k$-algebras but to arbitrary complete local commutative 
Noetherian rings with residue field $k$, thus obtaining more information about the 
connections between characteristic $p$ and characteristic
0 representations of $G$. Using Morita equivalence classes of algebras,
the case of $kG$ is tightly linked to understanding lifts of modules over arbitrary finite
dimensional $k$-algebras $\Lambda$. In particular, if $kG$ (or a block of $kG$) is
Morita equivalent to $\Lambda$ and $V$ corresponds to $V_\Lambda$ under this Morita
equivalence, then $V$ has a universal mod $p$ deformation ring $R(G,V)/pR(G,V)$ if and only if 
$V_\Lambda$ has a universal deformation ring $R(\Lambda,V_\Lambda)$. Moreover, 
$R(\Lambda,V_\Lambda)\cong R(G,V)/pR(G,V)$. The main advantage of employing this 
connection is that one can use powerful tools from the representation theory of finite dimensional 
algebras, such as Auslander-Reiten quivers, stable equivalences, and combinatorial
descriptions of modules, just to name a few. 
This approach has recently led to the solution of various open problems. For example, it was 
successfully used in \cite{3sim,bc5,bcs} to construct group representations whose universal 
deformation rings are not complete intersections, thus answering a question posed by Flach 
\cite{flach}.

In this paper, we concentrate on algebras of dihedral type of polynomial growth.
Algebras of dihedral type are certain symmetric algebras which
played an important role in Erdmann's classification  of all tame blocks of 
group algebras with dihedral defect groups up to Morita equivalence (see \cite{erd}).
A $k$-algebra is said to be of polynomial growth if its indecomposable modules 
of any given $k$-dimension $d$ can be parameterized using only a finite number $\mu(d)$ of 
one-parameter families and if there is a natural number $m$ such that $\mu(d)\le d^m$ for
all $d\ge 1$. 

Erdmann and Skowro\'{n}ski classified in \cite[Sect. 4]{erdskow} all basic algebras
of dihedral type of polynomial growth, by providing the quiver and relations of each such
algebra. They showed that there are precisely 8 possible quivers, each of which having at most 3 
vertices, and 12 Morita equivalence classes of algebras. It follows that the
stable Auslander-Reiten quiver of each such algebra $\Lambda$ has precisely one
non-periodic component of the form $\mathbb{Z}\tilde{A}_{p,q}$ 
where $(p,q)\in\{(1,1),(3,1),(3,3)\}$, depending on whether the quiver of $\Lambda$
has 1, 2, or 3 vertices. Moreover, there are up to two
3-tubes and infinitely many 1-tubes.
Here $\tilde{A}_{p,q}$ denotes the quiver
$$\tilde{A}_{p,q} = \quad \vcenter
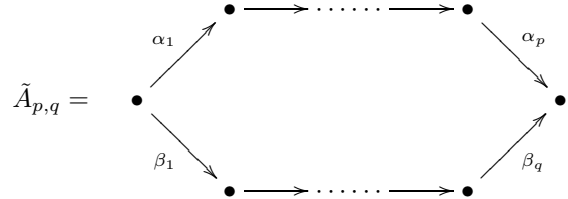
{\xymatrix{
&\bullet\ar[r]&\cdots\cdots\ar[r]&\bullet\ar[rd]^{\alpha_p}\\
\bullet\ar[ru]^{\alpha_1}\ar[rd]_{\beta_1}&&&&\bullet\\
&\bullet\ar[r]&\cdots\cdots\ar[r]&\bullet\ar[ru]_{\beta_q}
}}$$

A summary of our main results is as follows; for more precise statements, see Propositions 
\ref{prop:nonperiodic}, \ref{prop:3tubes} and \ref{prop:1tubes}.
Note that $\Omega$ stands for the syzygy functor (see e.g. \cite[pp. 124--126]{ARS}).

\begin{thm}
\label{thm:mainudr}
Let $\Lambda=kQ/I$ be an algebra of dihedral type of polynomial growth, and
let $\mathfrak{C}$ be a connected component of the stable Auslander-Reiten quiver of 
$\Lambda$. 
\begin{itemize}
\item[(i)] If $\mathfrak{C}$ is a non-periodic component, then $\mathfrak{C}$ is stable under
$\Omega$ and the stable endomorphism ring of every module $V$ belonging to $\mathfrak{C}$ 
is isomorphic to $k$. If $\mathfrak{C}$ is of the form $\mathbb{Z}\tilde{A}_{1,1}$, then there is 
precisely one $\Omega$-orbit of $\Lambda$-modules $V$ in $\mathfrak{C}$ and
$R(\Lambda,V)\cong \Lambda$.
If $\mathfrak{C}$ is of the form $\mathbb{Z}\tilde{A}_{3,1}$, then there are precisely two 
$\Omega$-orbits of $\Lambda$-modules $V$ in $\mathfrak{C}$ and
$R(\Lambda,V)\cong k[[t]]/(t^2)$. If $\mathfrak{C}$ is of the form $\mathbb{Z}\tilde{A}_{3,3}$,
then there are precisely three $\Omega$-orbits of $\Lambda$-modules $V$ in $\mathfrak{C}$ and
$R(\Lambda,V)$ is isomorphic to either $k$ or $k[[t]]/(t^2)$.
\item[(ii)] Suppose $\mathfrak{C}$ is a $3$-tube. If $\Omega(\mathfrak{C})= \mathfrak{C}$,
then the only modules in $\mathfrak{C}$ whose stable endomorphism rings are isomorphic to $k$
are the modules 
$V$ at the boundary of $\mathfrak{C}$ and $R(\Lambda,V)\cong k$.
If $\Omega(\mathfrak{C})\neq \mathfrak{C}$, then there are precisely three $\Omega$-orbits 
of modules $V$ in $\mathfrak{C}\cup\Omega(\mathfrak{C})$ whose stable endomorphism rings 
are isomorphic to $k$ and $R(\Lambda,V)$ is isomorphic to either $k$ or $k[[t]]$. 
\item[(iii)] Suppose $\mathfrak{C}$ is a $1$-tube. If $\Omega(\mathfrak{C})= \mathfrak{C}$,
then no module in $\mathfrak{C}$ has a stable endomorphism ring isomorphic to $k$.
If $\Omega(\mathfrak{C})\neq \mathfrak{C}$, then the only module in 
$\mathfrak{C}$ whose stable endomorphism ring is isomorphic to $k$
is the module $V$ at the boundary of $\mathfrak{C}$ and $R(\Lambda,V)\cong k[[t]]$.
\end{itemize}
\end{thm}

The main steps to prove Theorem \ref{thm:mainudr} are as follows: We use that for each algebra
$\Lambda$ of dihedral type of polynomial growth, $\Lambda/\mathrm{soc}(\Lambda)$
is a string algebra, which means that all non-projective indecomposable $\Lambda$-modules
are given combinatorially by string and band modules (see \cite{buri}). We use the description
of the $\Lambda$-module homomorphisms between string and band modules given in \cite{krau}
to find all $\Lambda$-modules $V$ whose stable endomorphism rings are isomorphic to $k$.
Since by \cite{bv} the universal deformation ring of $V$ is preserved by the syzygy functor
$\Omega$, it turns out that we only have to consider finitely many isomorphism classes of 
string modules $V$ whose stable endomorphism rings are isomorphic to $k$. 
Moreover, there is precisely one band attached to each $\Lambda$, which means that
we only have to determine for which $\mu\in k^*$ the 
stable endomorphism ring of the band module $V$ which is associated to $\mu$ and
which lies at the boundary of its 1-tube is isomorphic to $k$.
We then compute the $k$-dimension of
$\mathrm{Ext}^1_\Lambda(V,V)$ for each obtained string or band module
$V$ and use this to determine the 
isomorphism type of the universal deformation ring $R(\Lambda,V)$.

Let $G$ be a finite group. 
In \cite[Thm. 4.1]{erdskow}, the blocks $B$ of $kG$ which are representation
infinite of polynomial growth were determined. It was shown that $B$ is such a block if and only
if $\mathrm{char}(k)=2$ and the defect groups of $B$ are Klein four groups. In particular, $B$
is an algebra of dihedral type of polynomial growth. The determination of the universal 
deformation rings in Theorem \ref{thm:mainudr} enables us to give a characterization
of representation infinite blocks of polynomial growth in terms of universal deformation rings
(see Corollary \ref{cor:blockcase}). Note that the full universal deformation rings of modules
for blocks with Klein four defect groups were determined in \cite{bl}.

The paper is organized as follows. In \S\ref{s:prelim}, we recall the definitions of deformations 
and deformation rings. In \S\ref{s:dihedral}, we describe the quivers and relations of all
basic algebras $\Lambda$ of dihedral type of polynomial growth provided in
\cite[Sect. 4]{erdskow} and prove Theorem \ref{thm:mainudr}. In \S\ref{s:stringband}, we
give a brief introduction to string and band modules as given in \cite{buri}.
 
Part of this paper constitutes the Ph.D. thesis of the second author under the supervision
of the first author \cite{talbott}.


\section{Versal and universal deformation rings}
\label{s:prelim}
\setcounter{equation}{0}

In this section, we give a brief introduction to versal and universal deformation rings and deformations. For more background material, we refer the reader to \cite{bv}.

Let $k$ be a field of arbitrary characteristic. Let $\hat{\mathcal{C}}$ be the 
category of all complete local commutative Noetherian $k$-algebras with residue field $k$. The 
morphisms in $\hat{\mathcal{C}}$ are continuous $k$-algebra homomorphisms which induce 
the identity map on $k$. 

Suppose $\Lambda$ is a finite dimensional $k$-algebra and $V$ is a finitely generated
$\Lambda$-module. A lift of $V$ over an object $R$ in $\hat{\mathcal{C}}$ is a finitely generated 
$R\otimes_k \Lambda$-module $M$ which is free over $R$ together with a $\Lambda$-module 
isomorphism $\phi:k\otimes_R M\to V$. Two lifts 
$(M,\phi)$ and $(M',\phi')$ of $V$ over $R$ are isomorphic if there exists an 
$R\otimes_k\Lambda$-module isomorphism  $f:M\to M'$ such that 
$\phi'\circ(k\otimes_R f) = \phi$. The isomorphism class of a lift $(M,\phi)$ of $V$ 
over $R$ is denoted by $[M,\phi]$ and called a deformation of $V$ over $R$. 
We denote the set of all such deformations of $V$ over $R$ by $\mathrm{Def}_{\Lambda}(V,R)$. 
The deformation functor 
$$\hat{F}_V:\hat{\mathcal{C}} \to \mathrm{Sets}$$ 
is the covariant functor which
sends an object $R$ in $\hat{\mathcal{C}}$ to $\mathrm{Def}_{\Lambda}(V,R)$ and a morphism 
$\alpha:R\to R'$ in $\hat{\mathcal{C}}$ to the map $\mathrm{Def}_{\Lambda}(V,R) \to
\mathrm{Def}_{\Lambda}(V,R')$ defined by $[M,\phi]\mapsto [R'\otimes_{R,\alpha} M,\phi_\alpha]$, 
where  $\phi_\alpha=\phi$ after identifying $k\otimes_{R'}(R'\otimes_{R,\alpha} M)$ with 
$k\otimes_R M$.

Suppose there exists an object $R(\Lambda,V)$ in $\hat{\mathcal{C}}$ and a deformation 
$[U(\Lambda,V),\phi_U]$ of $V$ over $R(\Lambda,V)$ with the following property:
For each $R$ in $\hat{\mathcal{C}}$ and for each lift $(M,\phi)$ of $V$ over $R$ there exists 
a morphism $\alpha:R(\Lambda,V)\to R$ in $\hat{\mathcal{C}}$ such that 
$\hat{F}_V(\alpha)([U(\Lambda,V),\phi_U])=[M,\phi]$, and moreover $\alpha$ is unique if 
$R$ is the ring of dual numbers $k[\epsilon]/(\epsilon^2)$. Then $R(\Lambda,V)$ is called the 
versal deformation ring of $V$ and 
$[U(\Lambda,V),\phi_U]$ is called the versal deformation of $V$. If the morphism $\alpha$ is
unique for all $R$ and all lifts $(M,\phi)$ of $V$ over $R$, 
then $R(\Lambda,V)$ is called the universal deformation ring of $V$ and $[U(\Lambda,V),\phi_U]$ 
is called the universal deformation of $V$. In other words, $R(\Lambda,V)$ is universal if and only 
if $R(\Lambda,V)$ represents the functor $\hat{F}_V$ in the sense that $\hat{F}_V$ is naturally 
isomorphic to the Hom functor $\mathrm{Hom}_{\hat{\mathcal{C}}}(R(\Lambda,V),-)$. 

Note that the above definition of deformations can be weakened as follows.
Given a lift $(M,\phi)$ of $V$ over a ring $R$ in 
$\hat{\mathcal{C}}$, define the corresponding weak deformation to be
the isomorphism class of $M$ as an $R\otimes_k\Lambda$-module, without taking into account 
the specific isomorphism $\phi:k\otimes_RM\to V$. 
In general, a weak deformation of $V$ over $R$ identifies more lifts than a deformation of $V$ 
over $R$ that respects the isomorphism $\phi$ of a representative $(M,\phi)$.
However, if $\Lambda$ is self-injective and the stable endomorphism ring 
$\underline{\mathrm{End}}_{\Lambda}(V)$ is isomorphic to $k$, these two definitions
of deformations coincide (see  \cite[Thm. 2.6]{bv}).

It is straightforward to check that every finitely generated $\Lambda$-module $V$ has a 
versal deformation ring and that this versal deformation ring is universal if the endomorphism
ring $\mathrm{End}_\Lambda(V)$ is isomorphic to $k$ (see \cite[Prop. 2.1]{bv}). 
Moreover, Morita equivalences preserve versal deformation rings (see \cite[Prop. 2.5]{bv}).
If $\Lambda$ is self-injective, we obtain the following result, where $\Omega$ denotes the 
syzygy functor (see e.g. \cite[pp. 124--126]{ARS}).

\begin{thm}
\label{thm:udrvelez}
{\rm (\cite[Thm. 2.6]{bv})}
Let $\Lambda$ be a finite dimensional self-injective $k$-algebra,
and suppose $V$ is a finitely generated $\Lambda$-module whose stable endomorphism ring 
$\underline{\mathrm{End}}_{\Lambda}(V)$ is isomorphic to $k$.  
\begin{itemize}
\item[(i)] The module $V$ has a universal deformation ring $R(\Lambda,V)$. 
 \item[(ii)] If $P$ is a finitely generated projective $\Lambda$-module, then
$\underline{\mathrm{End}}_{\Lambda}(V\oplus P)\cong k$ and 
$R(\Lambda,V)\cong R(\Lambda,V\oplus P)$.
\item[(iii)] If $\Lambda$ is moreover a
Frobenius algebra, then $\underline{\mathrm{End}}_{\Lambda}(\Omega(V))\cong k$ and 
$R(\Lambda,V)\cong R(\Lambda,\Omega(V))$.
\end{itemize}
\end{thm}


\section{Algebras of dihedral type of polynomial growth}
\label{s:dihedral}

In this section we consider all algebras of dihedral type which are of polynomial growth. 
Let $k$ be an algebraically closed field of arbitrary characteristic. 
Algebras of dihedral type are certain symmetric algebras which
play an important role in Erdmann's classification of all tame blocks of 
group algebras with dihedral defect groups up to Morita equivalence (see \cite{erd}). Since by 
\cite{brauer2}, these tame blocks have at most three isomorphism classes of simple modules, 
this is also the case for the algebras of dihedral type. However, the class of algebras
of dihedral type strictly contains the Morita equivalence classes of these tame blocks.
A $k$-algebra is said to be of polynomial growth if it is tame, i.e. its indecomposable modules 
of any given $k$-dimension $d$ can  (up to finitely many exceptions)
be parameterized using only a finite number $\mu(d)$ of 
one-parameter families, and if there is a natural number $m$ such that $\mu(d)\le d^m$ for
all $d\ge 1$. Erdmann and Skowro\'{n}ski classified in \cite[Sect. 4]{erdskow} all algebras
of dihedral type of polynomial growth up to Morita equivalence. In particular, they showed
in \cite[Thm. 4.1]{erdskow} that the only tame blocks with dihedral defect groups which are of 
polynomial growth have Klein four defect groups. The latter fall into three distinct Morita
equivalence classes of algebras with either exactly one or exactly three isomorphism classes
of simple modules. 

Let $\Lambda=kQ/I$ be an algebra of dihedral type of polynomial growth. Then $Q$ is
one of 8 possible quivers listed in Figure \ref{fig:quivers}. For each $Q$, Erdmann and
Skowro\'{n}ski provide either one or two ideals $I$ of $kQ$ to obtain a complete list of Morita 
equivalence classes of algebras of dihedral type of polynomial growth. These ideals $I$ 
are listed in Figure \ref{fig:relations}.

\begin{figure}[ht] \caption{\label{fig:quivers} Quivers of algebras of dihedral type of 
polynomial growth.}
$$\begin{array}{r@{}ccr@{}c}
1\mathcal{A}=&
\xymatrix @R=-.2pc {
&0&\\
&\ar@(ul,dl)_{\alpha} \bullet \ar@(ur,dr)^{\beta}}
&&
2\mathcal{A}=& 
\xymatrix @R=-.2pc {
&0&1\\
&\ar@(ul,dl)_{\alpha} \bullet \ar@<.8ex>[r]^{\beta} \ar@<.9ex>[r];[]^{\gamma}
& \bullet }
\\[4em]
3\mathcal{A}=&
\xymatrix @R=-.2pc {
&1&\\
0\; \bullet \ar@<.8ex>[r]^(.56){\beta} \ar@<1ex>[r];[]^(.44){\gamma}
& \bullet \ar@<.8ex>[r]^(.44){\delta} \ar@<1ex>[r];[]^(.56){\eta} & \bullet\; 2}
&&
 \qquad3\mathcal{B}= &
\xymatrix @R=-.2pc {
&0&1&\\
&\ar@(ul,dl)_{\alpha} \bullet \ar@<.8ex>[r]^{\beta} \ar@<.9ex>[r];[]^{\gamma}
& \bullet \ar@<.8ex>[r]^(.46){\delta} \ar@<.9ex>[r];[]^(.54){\eta} & \bullet\;2}
\\[3em]
 3\mathcal{D}=&
\xymatrix @R=-.2pc {
&0&1&2\\
&\ar@(ul,dl)_{\alpha} \bullet \ar@<.8ex>[r]^{\beta} \ar@<.9ex>[r];[]^{\gamma}
& \bullet \ar@<.8ex>[r]^(.46){\delta} \ar@<.9ex>[r];[]^(.54){\eta} & \bullet \ar@(ur,dr)^{\xi}}
&&
3\mathcal{K}=&
\vcenter{\xymatrix @R1.8pc {
 0\;\bullet \ar@<.7ex>[rr]^{\beta} \ar@<.8ex>[rr];[]^{\gamma}\ar@<.7ex>[rdd]^{\kappa} \ar@<.8ex>[rdd];[]^{\lambda}
&&\bullet\ar@<.7ex>[ldd]^{\delta} \ar@<.8ex>[ldd];[]^{\eta}\;1\\&&\\ 
&\unter{\mbox{\normalsize $\bullet$}}{\mbox{\normalsize $2$}}& }}
\\[6em]
3\mathcal{L}=&
\vcenter{\xymatrix @R=.01pc {
\ar@(ul,dl)_{\alpha} \unter{\mbox{\normalsize $0$}}{\mbox{\normalsize $\bullet$}}
\ar[rr]^{\beta}  &&\unter{\mbox{\normalsize $1$}}{\mbox{\normalsize $\bullet$}}
\ar[ldddddddd]^{\delta} 
\\&&\\&&\\&&\\&&\\&&\\&&\\&&\\ 
&
\unter{\mbox{\normalsize $\bullet$}}{\mbox{\normalsize $2$}}
\ar[uuuuuuuul]^{\lambda}
&}}
&&
3\mathcal{Q}=&
\vcenter{\xymatrix @R=.01pc {
\ar@(ul,dl)_{\alpha} \unter{\mbox{\normalsize $0$}}{\mbox{\normalsize $\bullet$}} 
\ar[rr]^{\beta}  &&\unter{\mbox{\normalsize $1$}}{\mbox{\normalsize $\bullet$}}
\ar[ldddddddd]^{\delta} \ar@(ur,dr)^{\rho} 
\\&&\\&&\\&&\\&&\\&&\\&&\\&&\\ 
&
\unter{\mbox{\normalsize $\bullet$}}{\mbox{\normalsize $2$}}
\ar[uuuuuuuul]^{\lambda}
& }}
\end{array}
$$
\end{figure}

\begin{figure}[ht]
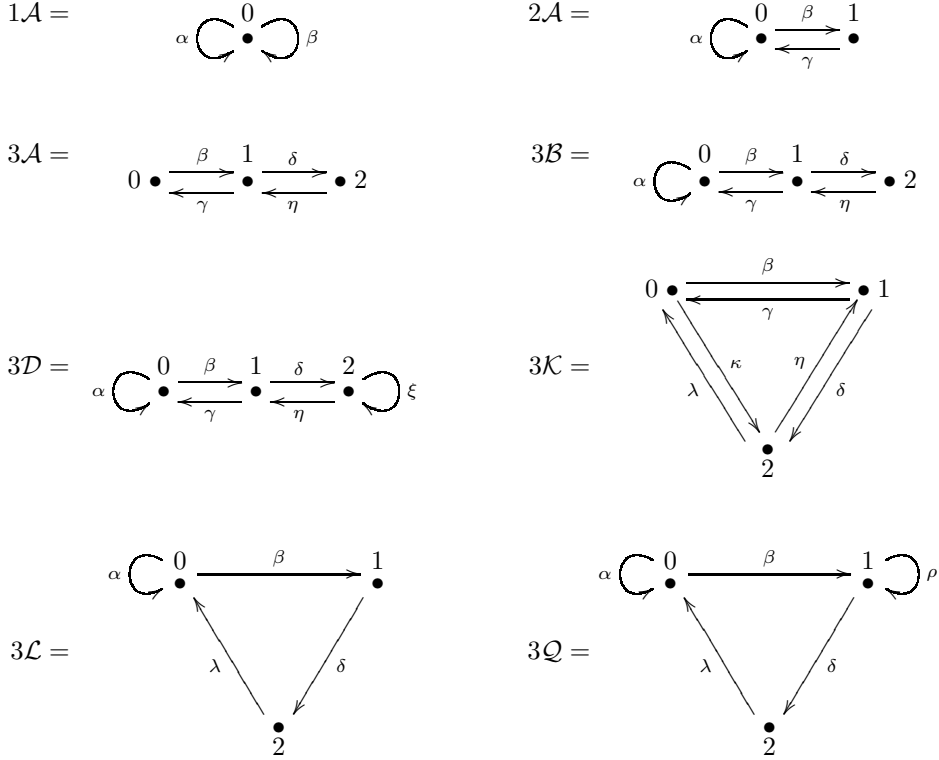
 \caption{\label{fig:relations} Algebras of dihedral type of 
polynomial growth.}
\begin{eqnarray*}
D(1)_0&=&k[1\mathcal{A}]/\langle \alpha^2,\beta^2,\alpha\beta-\beta\alpha\rangle,\\
D(1)_1&=&k[1\mathcal{A}]/\langle \alpha^2,\beta^2-\alpha\beta,\alpha\beta-\beta\alpha\rangle
\qquad \mbox{ and }\mathrm{char}(k)=2,\\
D(2\mathcal{A})_0&=&k[2\mathcal{A}]/\langle \alpha^2,\beta\gamma,
\alpha\gamma\beta-\gamma\beta\alpha\rangle,\\
D(2\mathcal{A})_1&=&k[2\mathcal{A}]/\langle \alpha^2-\alpha\gamma\beta,\beta\gamma,
\alpha\gamma\beta-\gamma\beta\alpha\rangle
\qquad \mbox{ and }\mathrm{char}(k)=2,\\
D(3\mathcal{A})_1&=&k[3\mathcal{A}]/\langle \gamma\beta,\delta\eta,
\beta\gamma\eta\delta-\eta\delta\beta\gamma\rangle,\\
D(3\mathcal{A})_2&=&k[3\mathcal{A}]/\langle \gamma\eta,\delta\beta,
(\beta\gamma)^2-(\eta\delta)^2\rangle,\\
D(3\mathcal{B})_{2,1}&=&k[3\mathcal{B}]/
\langle \alpha\gamma,\beta\alpha,\gamma\eta,\delta\beta,
\alpha^2-\gamma\beta,\beta\gamma-(\eta\delta)^2\rangle,\\
D(3\mathcal{B})_{2,2}&=&k[3\mathcal{B}]/
\langle \alpha\gamma,\beta\alpha,\gamma\eta,\delta\beta,
\alpha^2-(\gamma\beta)^2,(\beta\gamma)^2-\eta\delta\rangle,\\
D(3\mathcal{D})_2&=&k[3\mathcal{D}]/\langle \alpha\gamma,\beta\alpha,\gamma\eta,\delta\beta,
\eta\xi,\xi\delta,\alpha^2-\gamma\beta,\beta\gamma-\eta\delta,\delta\eta-\xi^2\rangle,\\
D(3\mathcal{K})&=&k[3\mathcal{K}]/\langle \beta\lambda,\gamma\eta,\delta\beta,\eta\kappa,
\kappa\gamma,\lambda\delta,\beta\gamma-\eta\delta,
\gamma\beta-\lambda\kappa,\delta\eta-\kappa\lambda\rangle,\\
D(3\mathcal{L})&=&k[3\mathcal{L}]/\langle\alpha\lambda, \beta\alpha,
\alpha^2-(\lambda\delta\beta)^2,\delta(\beta\lambda\delta)^2\rangle,\\
D(3\mathcal{Q})&=&k[3\mathcal{Q}]/\langle \alpha\lambda,\beta\alpha,\delta\rho,\rho\beta,
\alpha^2-\lambda\delta\beta,\beta\lambda\delta-\rho^2\rangle.
\end{eqnarray*}
\end{figure}

For each algebra $\Lambda$ in Figure \ref{fig:relations}, 
$\Lambda/\mathrm{soc}(\Lambda)$ is a string algebra. In particular, all 
indecomposable non-projective $\Lambda$-modules are given combinatorially by 
string and band modules (see \cite{buri}). Moreover, the $\Lambda$-module homomorphisms 
between string and band modules have been explicitly described in \cite{krau}.
For the convenience of the reader, we give a brief introduction to string and band modules in
\S \ref{s:stringband}.

For each $\Lambda$, there exists a unique pair $(p,q)\in\{(1,1),(3,1),(3,3)\}$ and a unique
band $B$ such that the stable Auslander-Reiten quiver of $\Lambda$ 
consists of a non-periodic component of the form $\mathbb{Z}\tilde{A}_{p,q}$, one $p$-tube 
and one $q$-tube consisting of string modules, and for each $\mu\in k^*$ a 1-tube consisting
of band modules with $M(B,\mu,1)$ lying at the boundary of this 1-tube.

In the following subsections, we find for each $\Lambda$ in Figure \ref{fig:relations} 
all indecomposable modules whose 
stable endomorphism rings are isomorphic to $k$ and we determine for each such module its 
universal deformation ring. We organize the modules according to the components of the
stable Auslander-Reiten quiver of $\Lambda$ to which they belong.

\subsection{Non-periodic components}

In this subsection, we show that the stable endomorphism ring of every module which belongs 
to the non-periodic component of the stable Auslander-Reiten quiver is isomorphic to $k$
and determine the universal deformation ring of each of these modules.

\begin{prop}
\label{prop:nonperiodic}
Let $\Lambda=kQ/I$ be a basic algebra of dihedral type of polynomial growth,
and let $\mathfrak{C}$ be the non-periodic component of the stable Auslander-Reiten quiver of 
$\Lambda$. Then the stable endomorphism ring of every module
belonging to $\mathfrak{C}$ is isomorphic to $k$.
\begin{itemize}
\item[(i)] If $\Lambda=D(1)_c$ for $c\in\{0,1\}$, then  there is precisely one $\Omega$-orbit of 
modules in $\mathfrak{C}$ and $R(\Lambda,V)\cong \Lambda$ for every module $V$ belonging
to $\mathfrak{C}$.
\item[(ii)] If $\Lambda=D(2\mathcal{A})_c$ for $c\in\{0,1\}$, then  there are precisely two 
$\Omega$-orbits of modules in $\mathfrak{C}$ and $R(\Lambda,V)\cong k[[t]]/(t^2)$ for every 
module $V$ belonging to $\mathfrak{C}$.
\item[(iii)] If $\Lambda\in\{D(3\mathcal{A})_1,D(3\mathcal{K})\}$, then there are precisely three
$\Omega$-orbits of modules in $\mathfrak{C}$ and $R(\Lambda,V)\cong k$ for every 
module $V$ belonging to $\mathfrak{C}$.
\item[(iv)] If $\Lambda\in \{D(3\mathcal{A})_2,D(3\mathcal{B})_{2,1},D(3\mathcal{B})_{2,2},
D(3\mathcal{D})_2,D(3\mathcal{L}),D(3\mathcal{Q})\}$, then there are precisely three
$\Omega$-orbits of modules in $\mathfrak{C}$ and $R(\Lambda,V)\cong k$ for
$V$ in precisely one of these $\Omega$-orbits and $R(\Lambda,V)\cong k[[t]]/(t^2)$
for $V$ in the remaining two $\Omega$-orbits. 
\end{itemize}
\end{prop}

\begin{proof}
Suppose first that $Q$ has a unique vertex, i.e. $\Lambda=D(1)_c$ for $c\in\{0,1\}$, where 
$c=1$ only occurs when $\mathrm{char}(k)=2$.
Then $\mathfrak{C}$ is of the form $\mathbb{Z}\tilde{A}_{1,1}$ and there is precisely
one $\Omega$-orbit of $\Lambda$-modules in $\mathfrak{C}$ represented by the simple 
module $S_0$ corresponding to the unique vertex in $Q$. To prove part (i), we can by Theorem 
\ref{thm:udrvelez} restrict to the case when $V=S_0$.
Note that $\Lambda$ is a ring in $\hat{\mathcal{C}}$ and that $U=\Lambda$ defines a lift
of $S_0$ over $\Lambda$, where a simple tensor $a\otimes b\in \Lambda\otimes_k\Lambda$
acts on $U$ as multiplication by $a b$. Let $R$ be an arbitrary ring in $\hat{\mathcal{C}}$ 
and let $M$ be a lift of $S_0$ over $R$. Then the $\Lambda$-action on $M=R$ is given by a 
continuous $k$-algebra homomorphism $\tau_M:\Lambda\to R$ which induces the identity
on the residue field $k$, i.e. $\tau_M$ is a morphism in $\hat{\mathcal{C}}$. Moreover, since
$M$ is free of rank 1 over $R$, $\mathrm{Aut}_{R\otimes_k\Lambda}(M)\cong
\mathrm{Aut}_R(M)\cong \mathrm{GL}_1(R)$, which implies that $\tau_M$ is unique. Since 
$M=R\otimes_{\Lambda,\tau_M}U$, it follows that $U$ defines a universal lift of $S_0$ over 
$\Lambda$. In other words, $R(\Lambda,S_0)\cong \Lambda$, which proves part (i).

Next suppose that $Q$ has two vertices, i.e. $\Lambda=D(2\mathcal{A})_c$ for $c\in\{0,1\}$, 
where $c=1$ only occurs when $\mathrm{char}(k)=2$. 
Then $\mathfrak{C}$ is of the form $\mathbb{Z}\tilde{A}_{3,1}$ and there 
are precisely two $\Omega$-orbits of $\Lambda$-modules in $\mathfrak{C}$ represented by the 
simple module $S_0$ and the string module $M(\beta)$. To prove part (ii), we can by Theorem 
\ref{thm:udrvelez} restrict to the case when $V\in\{S_0,M(\beta)\}$. We have
$\mathrm{End}_\Lambda(V)\cong k\cong \mathrm{Ext}^1_\Lambda(V,V)$, which 
implies that $R(\Lambda,V)$ is isomorphic to a quotient algebra of $k[[t]]$. There exists a non-split
short exact sequence of $\Lambda$-modules
$$0 \to V \xrightarrow{\iota} M \xrightarrow{\tau} V \to 0$$
where $M=M(\alpha)$ (resp. $M=M(\beta\alpha^{-1}\beta^{-1})$) if $V=S_0$ (resp. $V=M(\beta)$).
This means that $M$ defines a non-trivial lift of $V$ over $k[[t]]/(t^2)$ where the action of $t$ is 
given by $\iota\circ\tau$. Hence there exists a unique surjective $k$-algebra homomorphism 
$\psi:R(\Lambda,V)\to k[[t]]/(t^2)$ in $\hat{\mathcal{C}}$ corresponding to the deformation 
defined by $M$. We need to show that $\psi$ is an isomorphism.
Suppose this is false. Then there exists a surjective $k$-algebra homomorphism
$\psi_0:R(\Lambda,V)\to k[[t]]/(t^3)$ in $\hat{\mathcal{C}}$
such that $\pi\circ\psi_0=\psi$ where $\pi:k[[t]]/(t^3)\to
k[[t]]/(t^2)$ is the natural projection. Let $M_0$ be a $k[[t]]/(t^3)\otimes_k\Lambda$-module
which defines a lift of $V$ over $k[[t]]/(t^3)$ corresponding to $\psi_0$.
Because $M_0/t^2M_0\cong M$ and $t^2M_0\cong V$, we obtain 
a non-split short exact sequence of $k[[t]]/(t^3)\otimes_k\Lambda$-modules
$$0\to V \to M_0\to M\to 0.$$
Since $\mathrm{Ext}^1_{\Lambda}(M,V)=0$, this sequence splits as a sequence
of $\Lambda$-modules. Hence $M_0=V\oplus M$ as $\Lambda$-modules.
Writing elements of $M_0$ as $(v,x)$ where $v\in V$ and $x\in M$, the action of $t$
on $M_0$ is given by $t\,(v,x)=(\sigma(x),t\, x)$, where $\sigma:M\to V$ is a 
surjective $\Lambda$-module homomorphism. Since for each such $\sigma$ 
its kernel is equal to $tM$,
it follows that $t^2(v,x)=(\sigma(tx),t^2x)=(0,0)$ for all
$v\in V$ and $x\in M$. But this is a contradiction to $t^2M_0\cong V$. Thus $\psi$ is a 
$k$-algebra isomorphism and $R(\Lambda,V)\cong k[[t]]/(t^2)$, which proves part (ii).

Next suppose $Q$ has precisely three vertices. Then $\mathfrak{C}$ is of the form 
$\mathbb{Z}\tilde{A}_{3,3}$ and there are precisely three $\Omega$-orbits of 
$\Lambda$-modules in $\mathfrak{C}$ represented by $V_0,V_1,V_2$ as follows:
\begin{itemize}
\item If $\Lambda\in\{D(3\mathcal{A})_1, D(3\mathcal{A})_2\}$, then 
$V_0=S_1, V_1=M(\beta), V_2=M(\eta)$. 
\item If $\Lambda=D(3\mathcal{B})_{2,1}$, then $V_0=S_1, V_1=S_0, V_2=M(\eta)$.
\item If $\Lambda=D(3\mathcal{B})_{2,2}$, then $V_0=M(\gamma\delta^{-1}),V_1=S_0,
V_2=M(\gamma)$.
\item If $\Lambda\in\{D(3\mathcal{D})_2,D(3\mathcal{K}))\}$, then $V_0=S_1, V_1=S_0,
V_2=S_2$.
\item If $\Lambda=D(3\mathcal{L})$, then $V_0=M(\beta),V_1=S_0,V_2=M(\delta\beta)$.
\item If $\Lambda=D(3\mathcal{Q})$, then $V_0=M(\delta),V_1=S_0,V_2=S_1$.
\end{itemize}
In all cases, $\mathrm{End}_\Lambda(V_i)\cong k$ for $i\in\{0,1,2\}$ and
$\mathrm{Ext}^1_\Lambda(V_0,V_0)=0$. Hence $R(\Lambda,V_0)\cong k$.

If $\Lambda\in\{D(3\mathcal{A})_1, D(3\mathcal{K})\}$, then 
$\mathrm{Ext}^1_\Lambda(V_i,V_i)=0$ also when $i\in\{1,2\}$, implying
$R(\Lambda,V_i)\cong k$ for all $i\in\{0,1,2\}$, which proves part (iii).

Finally, suppose $\Lambda\in \{D(3\mathcal{A})_2,D(3\mathcal{B})_{2,1},D(3\mathcal{B})_{2,2},
D(3\mathcal{D})_2,D(3\mathcal{L}),D(3\mathcal{Q})\}$, and let $i\in\{1,2\}$. Then
$\mathrm{Ext}^1_\Lambda(V_i,V_i)\cong k$, which implies that
$R(\Lambda,V_i)$ is isomorphic to a quotient algebra of $k[[t]]$. 
To prove that $R(\Lambda,V_i)\cong k[[t]]/(t^2)$, we first show that $V_i$ has a non-trivial lift 
$M_i$ over $k[[t]]/(t^2)$. We define $M_i$ as follows:
\begin{itemize}
\item If $\Lambda= D(3\mathcal{A})_2$, let $M_1=M(\beta\gamma\beta)$ and
$M_2=M(\eta\delta\eta)$.
\item If $\Lambda=D(3\mathcal{B})_{2,1}$, let $M_1=M(\alpha)$ and $M_2=M(\eta\delta\eta)$.
\item If $\Lambda=D(3\mathcal{B})_{2,2}$, let $M_1=M(\alpha)$ and
$M_2=M(\gamma\beta\gamma)$.
\item If $\Lambda=D(3\mathcal{D})_2$, let $M_1=M(\alpha)$ and $M_2=M(\xi)$.
\item If $\Lambda=D(3\mathcal{L})$, let $M_1=M(\alpha)$ and 
$M_2=M(\delta\beta\lambda\delta\beta)$.
\item If $\Lambda=D(3\mathcal{Q})$, let $M_1=M(\alpha)$ and $M_2=M(\rho)$.
\end{itemize}
Let $i\in\{1,2\}$. In all cases, there exists a non-split short exact sequence of $\Lambda$-modules
$$0 \to V_i \xrightarrow{\iota_i} M_i \xrightarrow{\tau_i} V_i \to 0$$
which means that $M_i$ defines a non-trivial lift of $V_i$ over $k[[t]]/(t^2)$ where the action of $t$ 
is given by $\iota_i\circ\tau_i$. Since $\mathrm{Ext}^1_\Lambda(M_i,V_i)=0$ and the only 
surjective $\Lambda$-module homomorphisms $\sigma_i:M_i\to V_i$ have kernel equal to 
$tM_i$, we can argue similarly to the case $D(2\mathcal{A})_c$ to show that 
$R(\Lambda,V_i)\cong k[[t]]/(t^2)$ for $i\in\{1,2\}$, proving part (iv).
\end{proof}

\subsection{$3$-tubes}

In this subsection, we find all modules belonging to 3-tubes of the stable Auslander-Reiten quiver
whose stable endomorphism rings are isomorphic to $k$ and determine their universal 
deformation rings.

\begin{prop}
\label{prop:3tubes}
Let $\Lambda=kQ/I$ be a basic algebra of dihedral type of polynomial growth,
let $\mathfrak{T}$ be a $3$-tube of the stable Auslander-Reiten quiver of $\Lambda$ consisting of 
string modules, and let $T_0$ be a module belonging to the boundary of $\mathfrak{T}$.
\begin{itemize}
\item[(i)] If $\Lambda=D(1)_c$ for $c\in\{0,1\}$, then $\mathfrak{T}$ does not exist.
\item[(ii)] If $\Lambda=D(2\mathcal{A})_c$ for $c\in\{0,1\}$, then $\mathfrak{T}$ is unique and
stable under $\Omega$. There is precisely one $\Omega$-orbit of modules in $\mathfrak{T}$
whose stable endomorphism rings are isomorphic to $k$ represented by $T_0$, and
$R(\Lambda,T_0)\cong k$.
\item[(iii)] If $\Lambda\in\{D(3\mathcal{A})_1,D(3\mathcal{K})\}$, then $\mathfrak{T}$ is one
of two $3$-tubes and $\mathfrak{T}$ is stable under $\Omega$. 
There is precisely one $\Omega$-orbit of modules in $\mathfrak{T}$
whose stable endomorphism rings are isomorphic to $k$ represented by $T_0$, and
$R(\Lambda,T_0)\cong k$.
\item[(iv)] If $\Lambda\in \{D(3\mathcal{A})_2,D(3\mathcal{B})_{2,1},D(3\mathcal{B})_{2,2},
D(3\mathcal{D})_2,D(3\mathcal{L}),D(3\mathcal{Q})\}$, then $\mathfrak{T}$ is one
of two $3$-tubes and $\Omega$ interchanges these two $3$-tubes.
There are precisely three $\Omega$-orbits of modules in $\mathfrak{T}\cup\Omega(\mathfrak{T})$
whose stable endomorphism rings are isomorphic to $k$ represented by $T_0$,
by a successor $T_1$ of $T_0$, and by a successor $T_2$ of $T_1$ which does not lie
in the $\Omega$-orbit of $T_0$. Moreover, $R(\Lambda,T_0)\cong k\cong R(\Lambda,T_1)$
and $R(\Lambda,T_2)\cong k[[t]]$. 
\end{itemize}
\end{prop}

\begin{proof}
If $\Lambda=D(1)_c$ then there are precisely two maximal directed strings, which means that
the stable Auslander-Reiten quiver of $\Lambda$ does not contain any 3-tubes.

Suppose now that $\Lambda=D(2\mathcal{A})_c$ for $c\in\{0,1\}$, 
where $c=1$ only occurs when $\mathrm{char}(k)=2$. Then $\mathfrak{T}$ is unique
and hence stable under $\Omega$. Using hooks and cohooks (see \cite[pp. 166--174]{buri}), 
we see that the $\Omega$-orbits of all $\Lambda$-modules in $\mathfrak{T}$ are represented 
by $S_1$, which lies at the boundary of $\mathfrak{T}$, and by
\begin{eqnarray*}
T_{1,j} &=&
M\left((\alpha^{-1}\gamma\beta)^{j-1}\alpha^{-1}\right),\\
T_{2,j}&=&M\left((\alpha^{-1}\gamma\beta)^{j-1}\alpha^{-1}\gamma\right),\\
T_{3,j}&=&M\left((\alpha^{-1}\gamma\beta)^j\alpha^{-1}\beta^{-1}\right)
\end{eqnarray*}
for all $j\ge 1$. For $i\in\{1,2,3\}$ and $j\ge 1$, there exists an endomorphism of $T_{i,j}$
whose image is isomorphic to $S_0$ and which does not factor through a projective 
$\Lambda$-module. Hence, the only modules in 
$\mathfrak{T}$ whose stable endomorphism rings are isomorphic to $k$ 
lie in the $\Omega$-orbit of $S_1$. Since $\mathrm{Ext}^1_\Lambda(S_1,S_1)=0$, we obtain that 
$R(\Lambda,T_0)\cong k$ for all $T_0$ at the boundary of $\mathfrak{T}$, which proves part (ii).

Next suppose $\Lambda$ has precisely three isomorphism classes of simple modules. 
Then $\mathfrak{T}$ is one of two 3-tubes.

If $\Lambda=D(3\mathcal{A})_1$ (resp. $\Lambda=D(3\mathcal{K})$), then $\mathfrak{T}$ is 
stable under $\Omega$ and the $\Omega$-orbit of the modules at the boundary of 
$\mathfrak{T}$ is represented by $T_0=S_1$ or $T_0=S_2$ (resp. by $T_0=M(\gamma)$ or 
$T_0=M(\lambda)$). Arguing similarly to the case $D(2\mathcal{A})_c$, we see that
the only modules in $\mathfrak{T}$ whose stable endomorphism rings are isomorphic to $k$ 
lie in the $\Omega$-orbit of $T_0$. Since $\mathrm{Ext}^1_\Lambda(T_0,T_0)=0$, we obtain that 
$R(\Lambda,T_0)\cong k$, which proves part (iii).

If $\Lambda\in \{D(3\mathcal{A})_2,D(3\mathcal{B})_{2,1},D(3\mathcal{B})_{2,2},
D(3\mathcal{D})_2,D(3\mathcal{L}),D(3\mathcal{Q})\}$, then $\Omega(\mathfrak{T})\neq
\mathfrak{T}$. The representatives $T_0,T_1,T_2$ from the statement of part (iv) can be
taken to be as follows:
\begin{itemize}
\item If $\Lambda=D(3\mathcal{A})_2$, then $T_0=S_0,
T_1=M(\gamma\delta^{-1}\eta^{-1}\delta^{-1}), T_2=M(\gamma\delta^{-1}\eta^{-1})$.
\item If $\Lambda=D(3\mathcal{B})_{2,1}$, then $T_0=S_2, T_1=M(\delta\gamma^{-1}),
T_2=M(\delta\gamma^{-1}\alpha\beta^{-1})$.
\item If $\Lambda=D(3\mathcal{B})_{2,2}$, then $T_0=M(\delta^{-1}), T_1=S_1,
T_2=M(\beta\alpha^{-1})$.
\item If $\Lambda=D(3\mathcal{D})_2$, then $T_0=M(\gamma^{-1}), 
T_1=M(\gamma^{-1}\alpha\beta^{-1}), T_2=M(\gamma^{-1}\alpha\beta^{-1}\eta\xi^{-1})$.
\item If $\Lambda=D(3\mathcal{L})$, then $T_0=S_2,T_1=M(\delta),
T_2=M(\delta\beta\alpha^{-1})$.
\item If $\Lambda=D(3\mathcal{Q})$, then $T_0=S_2, T_1=M(\delta\rho^{-1}), 
T_2=M(\delta\rho^{-1}\beta\alpha^{-1})$.
\end{itemize}
Arguing similarly to the case $D(2\mathcal{A})_c$, we see that
the only modules in $\mathfrak{T}$ whose stable endomorphism rings are isomorphic to $k$ 
lie in the $\Omega$-orbits of $T_0,T_1,T_2$. If $i\in\{0,1\}$ then
$\mathrm{Ext}^1_\Lambda(T_i,T_i)=0$ and hence $R(\Lambda,T_i)\cong k$.
On the other hand, $\mathrm{Ext}^1_\Lambda(T_2,T_2)=k$, which implies that 
$R(\Lambda,T_2)$ is isomorphic to a quotient algebra of $k[[t]]$. 
To prove that
$R(\Lambda,T_2)\cong k[[t]]$, it is therefore enough to prove that $T_2$ has a lift $L_2$ over
$k[[t]]$ such that $L_2/t^2L_2$ defines a non-trivial lift of $T_2$ over $k[[t]]/(t^2)$.
We define $L_2$ as follows:
\begin{itemize}
\item If $\Lambda=D(3\mathcal{A})_2$, let $L_2$ be the free $k[[t]]$-module of rank 4
with an ordered basis $\{B_0,B_1,B_2,B_3\}$ and define a $\Lambda$-module structure on $L_2$ 
by letting (the image of) each vertex (resp. arrow) $c$ in $Q$ act on $\{B_0,B_1,B_2,B_3\}$
as the following $4\times 4$ matrix $X_c$:
$X_{e_0}=E_{00}$, $X_{e_1}=E_{11}+E_{33}$, $X_{e_2}=E_{22}$,
$X_\beta=tE_{30}$, $X_\gamma=E_{01}$, $X_\delta=E_{21}$ and $X_\eta=E_{32}$. 
Here $E_{ji}$ denotes the $4\times 4$ matrix which sends $B_i$ to $B_j$ and all other basis 
elements to 0. 

\item If $\Lambda=D(3\mathcal{B})_{2,1}$, let $L_2$ be the free $k[[t]]$-module of rank 5
with an ordered basis $\{B_0,\ldots,B_4\}$ and define a $\Lambda$-module structure on 
$L_2$ by letting (the image of) each vertex (resp. arrow) $c$ in $Q$ act on $\{B_0,\ldots,B_4\}$
as the following $5\times 5$ matrix $X_c$:
$X_{e_0}=E_{22}+E_{33}$, $X_{e_1}=E_{11}+E_{44}$, $X_{e_2}=E_{00}$,
$X_\alpha=E_{23}$, $X_\beta=E_{43}$, $X_\gamma=E_{21}$, $X_\delta=E_{01}$ and 
$X_\eta=tE_{40}$. 

\item If $\Lambda=D(3\mathcal{B})_{2,2}$, let $L_2$ be the free $k[[t]]$-module of rank 3
with an ordered basis $\{B_0,B_1,B_2\}$ and let (the image of) each vertex (resp. arrow) $c$ in 
$Q$ act on $\{B_0,B_1,B_2\}$ as the following $3\times 3$ matrix $X_c$:
$X_{e_0}=E_{11}+E_{22}$, $X_{e_1}=E_{00}$, $X_{e_2}=0$,
$X_\alpha=E_{21}$, $X_\beta=E_{01}$, $X_\gamma=tE_{20}$, $X_\delta=0$ and 
$X_\eta=0$. 

\item If $\Lambda=D(3\mathcal{D})_2$, let $L_2$ be the free $k[[t]]$-module of rank 6
with an ordered basis $\{B_0,\ldots,B_5\}$ and let (the image of) each vertex (resp. arrow) $c$ in 
$Q$ act on $\{B_0,\ldots,B_5\}$ as the following $6\times 6$ matrix $X_c$:
$X_{e_0}=E_{11}+E_{22}$, $X_{e_1}=E_{00}+E_{33}$, $X_{e_2}=E_{44}+E_{55}$,
$X_\alpha=E_{12}$, $X_\beta=E_{32}$, $X_\gamma=E_{10}$, $X_\delta=tE_{50}$, 
$X_\eta=E_{34}$ and $X_{\xi}=E_{54}$. 

\item If $\Lambda=D(3\mathcal{L})$, let $L_2$ be the free $k[[t]]$-module of rank 4
with an ordered basis $\{B_0,\ldots,B_3\}$ and let (the image of) each vertex (resp. arrow) $c$ in 
$Q$ act on $\{B_0,\ldots,B_3\}$ as the following $4\times 4$ matrix $X_c$:
$X_{e_0}=E_{22}+E_{33}$, $X_{e_1}=E_{11}$, $X_{e_2}=E_{00}$,
$X_\alpha=E_{32}$, $X_\beta=E_{12}$,  $X_\delta=E_{01}$ and 
$X_\lambda=tE_{30}$. 

\item If $\Lambda=D(3\mathcal{Q})$, let $L_2$ be the free $k[[t]]$-module of rank 5
with an ordered basis $\{B_0,\ldots,B_4\}$ and let (the image of) each vertex (resp. arrow) $c$ in 
$Q$ act on $\{B_0,\ldots,B_4\}$ as the following $5\times 5$ matrix $X_c$:
$X_{e_0}=E_{33}+E_{44}$, $X_{e_1}=E_{11}+E_{22}$, $X_{e_2}=E_{00}$,
$X_\alpha=E_{43}$, $X_\beta=E_{23}$,  $X_\delta=E_{01}$, 
$X_\lambda=tE_{40}$  and $X_\rho=E_{21}$. 
\end{itemize}
In all cases, $L_2$ is a $k[[t]]\otimes_k\Lambda$-module which is free as a $k[[t]]$-module,
$L_2/tL_2\cong T_2$ as $\Lambda$-modules, and $L_2/t^2L_2$ defines a non-trivial
lift of $T_2$ over $k[[t]]/(t^2)$. This implies that $R(\Lambda,T_2)\cong k[[t]]$, proving part (iv).
\end{proof}

\subsection{$1$-tubes}

In this subsection, we find all modules belonging to 1-tubes of the stable Auslander-Reiten quiver
whose stable endomorphism rings are isomorphic to $k$ and determine their universal 
deformation rings. 

For each $\Lambda$ in Figure \ref{fig:relations}, there exists a unique band $B$.
For each $\mu\in k^*$, the band module $M(B,\mu,1)$ lies at the boundary of its 1-tube.
If $\Lambda=D(1)_c$ then $B=\beta\alpha^{-1}$, and we allow $\mu$ to lie
in $k\cup\{\infty\}$ by defining 
$M(\beta\alpha^{-1},0,m)=M\left(\alpha^{-1}(\beta\alpha^{-1})^{m-1}\right)$ and 
$M(\beta\alpha^{-1},\infty,m)=M\left((\beta\alpha^{-1})^{m-1}\beta\right)$
for all $m\in\mathbb{Z}^+$.
If $\Lambda=D(2\mathcal{A})_c$ then $B=\alpha\beta^{-1}\gamma^{-1}$, and we allow 
$\mu$ to lie in $k$ by defining $M(\alpha\beta^{-1}\gamma^{-1},0,m)=
M\left(\beta^{-1}\gamma^{-1}(\alpha\beta^{-1}\gamma^{-1})^{m-1}\right)$
for all $m\in\mathbb{Z}^+$.
By extending the values of $\mu$ this way, it follows that we can view the string modules 
belonging to 1-tubes as band modules. This allows us to treat all modules in 1-tubes in 
a uniform way.

\begin{prop}
\label{prop:1tubes}
Let $\Lambda=kQ/I$ be a basic algebra of dihedral type of polynomial growth, 
and let $B$ be the unique band for $\Lambda$. Let $m\in\mathbb{Z}^+$, and let
$\mu\in k\cup\{\infty\}$ $($resp. $\mu\in k$, resp.
$\mu\in k^*$$)$ according to $Q$ having $1$ $($resp. $2$, resp. $3$$)$ vertices.
Then the stable endomorphism ring of $M(B,\mu,m)$ is isomorphic to $k$ if and only if
$m=1$ and $\Omega(M(B,\mu,1))\not\cong M(B,\mu,1)$. More precisely:
\begin{itemize}
\item[(i)] If $\Lambda\in\{D(1)_0,D(2\mathcal{A})_0,D(3\mathcal{A})_1,D(3\mathcal{K})\}$ then 
$\underline{\mathrm{End}}_\Lambda(M(B,\mu,1))\cong k$ if and only if $\mathrm{char}(k)\neq 2$
and $\mu\in k^*$.
\item[(ii)]  If $\Lambda\in\{D(1)_1,D(2\mathcal{A})_1\}$ then 
$\underline{\mathrm{End}}_\Lambda(M(B,\mu,1))\cong k$ if and only if  $\mu\in k$.
\item[(iii)] If $\Lambda\in \{D(3\mathcal{A})_2,D(3\mathcal{B})_{2,2},D(3\mathcal{D})_2,
D(3\mathcal{L})\}$ then 
$\underline{\mathrm{End}}_\Lambda(M(B,\mu,1))\cong k$ if and only if  $\mu^2\neq-1$.
\item[(iv)] If $\Lambda\in\{D(3\mathcal{B})_{2,1},D(3\mathcal{Q})\}$ then 
$\underline{\mathrm{End}}_\Lambda(M(B,\mu,1))\cong k$ if and only if  $\mu^2\neq 1$.
\end{itemize}
In all cases, if $\underline{\mathrm{End}}_\Lambda(M(B,\mu,1))\cong k$ then
$R(\Lambda,M(B,\mu,1))\cong k[[t]]$.
\end{prop}

\begin{proof}
By \cite{krau}, it follows in all cases that if $M$ is a module in a 1-tube whose stable 
endomorphism ring is isomorphic to $k$, then $M$ has to belong to the boundary of the 1-tube.

Suppose first that $Q$ has a unique vertex, i.e. $\Lambda=D(1)_c$ for $c\in\{0,1\}$, where 
$c=1$ only occurs when $\mathrm{char}(k)=2$. Then $B=\beta\alpha^{-1}$ and 
$M(\beta)=M(B,\infty,1)$ always satisfies $\Omega(M(\beta))\cong
M(\beta)$ and $\underline{\mathrm{End}}_\Lambda(M(\beta))\not\cong k$.
If  $c=0$ and $\mathrm{char}(k)=2$ then $\Omega(M(B,\mu,1))\cong M(B,\mu,1)$, and 
$\underline{\mathrm{End}}_\Lambda(M(B,\mu,1))\not\cong k$ for all $\mu\in k$.
If $c=0$ and $\mathrm{char}(k)\neq 2$ then $\Omega(M(B,\mu,1))\cong M(B,-\mu,1)$, and 
$\underline{\mathrm{End}}_\Lambda(M(B,\mu,1))\cong k$ if and only if $\mu\in k^*$. 
If $c=1$ then $\Omega(M(B,\mu,1))\cong M(B,1-\mu,1)$, and 
$\underline{\mathrm{End}}_\Lambda(M(B,\mu,1))\cong k$ if and only if $\mu\in k$. 
In all cases when $\underline{\mathrm{End}}_\Lambda(M(B,\mu,1))\cong k$, we have
$\mathrm{Ext}^1_\Lambda(M(B,\mu,1),M(B,\mu,1))\cong k$, which means that
$R(\Lambda,M(B,\mu,1))$ is isomorphic to a quotient algebra of $k[[t]]$. Define
$L_\mu$ to be the free $k[[t]]$-module of rank 2 with an ordered basis $\{B_0,B_1\}$ and define a 
$\Lambda$-module structure on $L_\mu$ by letting (the image of) $\alpha$ (resp. $\beta$)
act on $\{B_0,B_1\}$ by $X_\alpha=E_{01}$ (resp. $X_\beta=(\mu+t)E_{01}$).
Here $E_{01}$ denotes the $2\times 2$ matrix which sends $B_1$ to $B_0$ and $B_0$ 
to 0. Then $L_\mu$ is a $k[[t]]\otimes_k\Lambda$-module which is free as a 
$k[[t]]$-module, $L_\mu/tL_\mu\cong M(B,\mu,1)$ as $\Lambda$-modules, and $L_\mu/t^2L_\mu$ 
defines a non-trivial lift of $M(B,\mu,1)$ over $k[[t]]/(t^2)$. This implies that 
$R(\Lambda,M(B,\mu,1))\cong k[[t]]$.

Next suppose that $Q$ has two vertices, i.e. $\Lambda=D(2\mathcal{A})_c$ for $c\in\{0,1\}$, 
where $c=1$ only occurs when $\mathrm{char}(k)=2$. Then $B=\alpha\beta^{-1}\gamma^{-1}$.
If  $c=0$ and $\mathrm{char}(k)=2$ then $\Omega(M(B,\mu,1))\cong M(B,\mu,1)$, and 
$\underline{\mathrm{End}}_\Lambda(M(B,\mu,1))\not\cong k$ for all $\mu\in k$.
If $c=0$ and $\mathrm{char}(k)\neq 2$ then $\Omega(M(B,\mu,1))\cong M(B,-\mu,1)$, and 
$\underline{\mathrm{End}}_\Lambda(M(B,\mu,1))\cong k$ if and only if $\mu\in k^*$. 
If $c=1$ then $\Omega(M(B,\mu,1))\cong M(B,1-\mu,1)$, and 
$\underline{\mathrm{End}}_\Lambda(M(B,\mu,1))\cong k$ if and only if $\mu\in k$. 
In all cases when $\underline{\mathrm{End}}_\Lambda(M(B,\mu,1))\cong k$, we have
$\mathrm{Ext}^1_\Lambda(M(B,\mu,1),M(B,\mu,1))\cong k$, which means that
$R(\Lambda,M(B,\mu,1))$ is isomorphic to a quotient algebra of $k[[t]]$. Define
$L_\mu$ to be the free $k[[t]]$-module of rank 4 with an ordered basis $\{B_0,B_1,B_2,B_3\}$ 
and define a $\Lambda$-module structure on $L_\mu$ by letting (the image of) $\alpha$ (resp. 
$\beta$, resp. $\gamma$) act on $\{B_0,B_1,B_2,B_3\}$ by $X_\alpha=(\mu+t)E_{01}$ 
(resp. $X_\beta=E_{21}$, resp. $X_\gamma=E_{02}$).
Here $E_{ji}$ denotes the $3\times 3$ matrix which sends $B_i$ to $B_j$ and all other basis
elements to 0. Then $L_\mu$ is a $k[[t]]\otimes_k\Lambda$-module which is free as a 
$k[[t]]$-module, $L_\mu/tL_\mu\cong M(B,\mu,1)$ as $\Lambda$-modules, and $L_\mu/t^2L_\mu$ 
defines a non-trivial lift of $M(B,\mu,1)$ over $k[[t]]/(t^2)$. This implies that 
$R(\Lambda,M(B,\mu,1))\cong k[[t]]$.

Next suppose that $Q$ has precisely three vertices and that $\Lambda\in
\{D(3\mathcal{A})_1,D(3\mathcal{K})\}$. If $\Lambda=D(3\mathcal{A})_1$ then
$B=\beta\gamma\delta^{-1}\eta^{-1}$, and if $\Lambda=D(3\mathcal{K})$ then
$B=\beta\kappa^{-1}\delta\gamma^{-1}\lambda\eta^{-1}$.
If $\mathrm{char}(k)=2$ then $\Omega(M(B,\mu,1))\cong M(B,\mu,1)$, and 
$\underline{\mathrm{End}}_\Lambda(M(B,\mu,1))\not\cong k$ for all $\mu\in k^*$.
If $\mathrm{char}(k)\neq 2$ then $\Omega(M(B,\mu,1))\cong M(B,-\mu,1)$, and 
$\underline{\mathrm{End}}_\Lambda(M(B,\mu,1))\cong k$ if and only if $\mu\in k^*$. 
In all cases when $\underline{\mathrm{End}}_\Lambda(M(B,\mu,1))\cong k$, it follows 
similarly to the case $D(2\mathcal{A})_c$ that $R(\Lambda,M(B,\mu,1))\cong k[[t]]$.

Next suppose that $Q$ has precisely three vertices and that $\Lambda\in
\{D(3\mathcal{A})_2,D(3\mathcal{B})_{2,2},D(3\mathcal{D})_2,
D(3\mathcal{L})\}$. If $\Lambda=D(3\mathcal{A})_2$ then
$B=\beta\gamma\delta^{-1}\eta^{-1}$, if $\Lambda=D(3\mathcal{B})_{2,2}$ then
$B=\alpha\beta^{-1}\gamma^{-1}$, if $\Lambda=D(3\mathcal{D})_2$ then
$B=\alpha\beta^{-1}\eta\xi^{-1}\delta\gamma^{-1}$, and if $\Lambda=D(3\mathcal{L})$
then $B=\alpha\beta^{-1}\delta^{-1}\lambda^{-1}$.
We have $\Omega(M(B,\mu,1))\cong M(B,-1/\mu,1)$, and 
$\underline{\mathrm{End}}_\Lambda(M(B,\mu,1))\cong k$ if and only if $\mu^2\neq -1$. 
If $\mu^2\neq -1$ then it follows 
similarly to the case $D(2\mathcal{A})_c$ that $R(\Lambda,M(B,\mu,1))\cong k[[t]]$.

Finally suppose that $Q$ has precisely three vertices and that $\Lambda\in
\{D(3\mathcal{B})_{2,1},D(3\mathcal{Q})\}$. If $\Lambda=D(3\mathcal{B})_{2,1}$ then
$B=\alpha\beta^{-1}\eta\delta\gamma^{-1}$, if $\Lambda=D(3\mathcal{Q})$
then $B=\alpha\beta^{-1}\rho\delta^{-1}\lambda^{-1}$.
We have $\Omega(M(B,\mu,1))\cong M(B,1/\mu,1)$, and 
$\underline{\mathrm{End}}_\Lambda(M(B,\mu,1))\cong k$ if and only if $\mu^2\neq 1$. 
If $\mu^2\neq 1$ then it follows 
similarly to the case $D(2\mathcal{A})_c$ that $R(\Lambda,M(B,\mu,1))\cong k[[t]]$.
\end{proof}

\subsection{Blocks which are representation infinite of polynomial growth}

Let $G$ be a finite group and let $B$ be a block of $kG$ which is representation infinite of
polynomial growth.
By \cite[Thm. 4.1]{erdskow}, this occurs if and only if $\mathrm{char}(k)=2$ and
the defect groups of $B$ are Klein four groups. In this case,
$B$ is Morita equivalent to either $D(1)_0$, or $D(3\mathcal{A})_1$, or $D(3\mathcal{K})$.
Using Propositions \ref{prop:nonperiodic}, \ref{prop:3tubes} and \ref{prop:1tubes}, we
obtain the following result.

\begin{cor}
\label{cor:blockcase}
Let $\Lambda=kQ/I$ be a basic algebra of dihedral type of polynomial growth.
Then $R(\Lambda,V)$ is finite dimensional over $k$ for all
finitely generated $\Lambda$-modules $V$ with $\underline{\mathrm{End}}_\Lambda(V)\cong k$
if and only if $\mathrm{char}(k)=2$ and $\Lambda$ is Morita equivalent to either
$D(1)_0$, or $D(2\mathcal{A})_0$, or $D(3\mathcal{A})_1$, or $D(3\mathcal{K})$.

In particular, suppose $G$ is a finite group and $B$ is a block of $kG$.
Then $B$ is representation infinite of polynomial growth if and only if
$B$ is Morita equivalent to an algebra $\Lambda$ of dihedral type of polynomial growth
and $R(\Lambda,V)$ is finite dimensional over $k$ for all finitely generated 
$\Lambda$-modules $V$ with $\underline{\mathrm{End}}_{\Lambda}(V)\cong k$.
\end{cor}

Note that $D(2\mathcal{A})_0$ is not Morita equivalent to a block of a group ring in characteristic 
2. The reason is that if it were Morita equivalent to a block $B$, then the defect groups
of $B$ would have to have cardinality 4 and hence they would have to be Klein four groups.
However, all blocks with Klein four defect groups are Morita equivalent to either $D(1)_0$, or 
$D(3\mathcal{A})_1$, or $D(3\mathcal{K})$.


\section{Appendix: String and band modules}
\label{s:stringband}

Let $k$ be an algebraically closed field, and let $\Lambda=kQ/I$ be a basic algebra of
dihedral type of polynomial growth as in Figure \ref{fig:relations}. Then 
$\overline{\Lambda}=\Lambda/\mathrm{soc}(\Lambda)=kQ/J$ is a string algebra.
In particular, all non-projective indecomposable $\Lambda$-modules are inflated from
string and band modules for $\overline{\Lambda}$. 
In this appendix, we give a brief introduction to these string and band modules.
For more details, see \cite{buri}.

For each arrow $\zeta$ in $Q$ with starting vertex $s(\zeta)$ and end vertex $e(\zeta)$,
we define a formal inverse $\zeta^{-1}$ with $s(\zeta^{-1})=e(\zeta)$ and $e(\zeta^{-1})=s(\zeta)$.
A word $w$ is a sequence $w_1\cdots w_n$, where $w_i$ is either an arrow or a formal inverse 
such that $s(w_i)=e(w_{i+1})$ for $1\leq i \leq n-1$. Define $s(w)=s(w_n)$, $e(w)=e(w_1)$ and 
$w^{-1}=w_n^{-1}\cdots w_1^{-1}$. For each vertex $u$ in $Q$, there is also an empty word
$1_u$ of length $0$ with $e(1_u)=u=s(1_u)$ and $(1_u)^{-1}=1_u$. 
Denote the set of all words by $\mathcal{W}$.

\begin{dfn}
\label{def:string}
Define an equivalence relation $\sim$ on
$\mathcal{W}$ by $w\sim w'$ if and only if $w=w'$ or $w^{-1}=w'$. 
A string is a representative $w\in{\mathcal{W}}$ of an equivalence class under 
$\sim_s$ with the following property: Either $w=1_u$ for some vertex $u$ in $Q$, or 
$w=w_1\cdots w_n$ for some $n\ge 1$
where $w_i \neq w_{i+1}^{-1}$ for $1\leq i\leq n-1$ and no subword of $w$ 
or its formal inverse belongs to $J$.

If $C=1_u$ for some vertex $u$ in $Q$, then the string module $M(1_u)$ is defined to be the 
simple $\Lambda$-module $S_u$ corresponding to $u$.
Let now $C=w_1\cdots w_n$ be a string of length $n\ge 1$, and
define $v(i)=e(w_{i+1})$ for $0\leq i\leq n-1$ and $v(n)=s(w_n)$. 
Then the string module $M(C)$ is defined to have an ordered $k$-basis $\{z_0,z_1,\ldots, z_n\}$ 
such that the $\Lambda$-action on $M(C)$ is given by letting (the images of) each vertex $u$ 
and each arrow $\zeta$ of $Q$ act on $\{z_0,z_1,\ldots, z_n\}$ as the following
$(n+1)\times(n+1)$ matrices $X_u$ and $X_\zeta$, respectively: 
$X_u$ sends $z_i$ to itself if $v(i)=u$ and $X_u$ sends $z_i$ to 0 otherwise;
whereas $X_\zeta$ sends $z_i$ to $z_{i-1}$ (resp. $z_{i+1}$) if $w_i=\zeta$ (resp. 
$w_{i+1}=\zeta^{-1}$) and $X_\zeta$ sends $z_i$ to 0 otherwise.
\end{dfn}

\begin{dfn}
\label{def:band}
For each $\Lambda=kQ/I$ as in Figure \ref{fig:relations}, there exists a unique band $B$ as 
follows:
\begin{itemize}
\item If $\Lambda=D(1)_c$, then $B=\beta\alpha^{-1}$.
\item If $\Lambda=D(2\mathcal{A})_c$, then $B=\alpha\beta^{-1}\gamma^{-1}$.
\item If $\Lambda\in\{D(3\mathcal{A})_1,D(3\mathcal{A})_2\}$, then 
$B=\beta\gamma\delta^{-1}\eta^{-1}$.
\item If $\Lambda=D(3\mathcal{B})_{2,1}$, then $B=\alpha\beta^{-1}\eta\delta\gamma^{-1}$.
\item If $\Lambda=D(3\mathcal{B})_{2,2}$, then $B=\alpha\beta^{-1}\gamma^{-1}$.
\item If $\Lambda=D(3\mathcal{D})_2$, then $B=\alpha\beta^{-1}\eta\xi^{-1}\delta\gamma^{-1}$.
\item If $\Lambda=D(3\mathcal{K})$, then 
$B=\beta\kappa^{-1}\delta\gamma^{-1}\lambda\eta^{-1}$
\item If $\Lambda=D(3\mathcal{L})$, then $B=\alpha\beta^{-1}\delta^{-1}\lambda^{-1}$.
\item If $\Lambda=D(3\mathcal{Q})$, then $B=\alpha\beta^{-1}\rho\delta^{-1}\lambda^{-1}$.
\end{itemize}
Let $B=w_1\cdots w_n$ be the band for $\Lambda=kQ/I$ as above, and
define $v(i)=e(w_{i+1})$ for $0\leq i\leq n-1$. Then for each integer $m>0$ and each 
$\mu\in k^*$, the band module $M(B,\mu,m)$ is defined to have an ordered $k$-basis 
\begin{equation}
\label{eq:orderedband}
\{z_{0,1},z_{0,2},\ldots, z_{0,m},z_{1,1},\ldots,z_{1,m},\ldots, z_{n-1,1},\ldots,z_{n-1,m}\}
\end{equation}
such that the $\Lambda$-action on $M(B,\mu,m)$ is given by letting (the images of) each vertex 
$u$ and each arrow $\zeta$ of $Q$ act on the basis in (\ref{eq:orderedband}) as the following
$nm\times nm$ matrices $X_u$ and $X_\zeta$, repsectively: 
$X_u$ sends $z_{i,j}$ to itself if $v(i)=u$ and $X_u$ sends $z_{i,j}$ to 0 otherwise;
whereas $X_\zeta$ sends $z_{i,j}$ to $\mu\,z_{0,j}+z_{0,j+1}$ (resp. $z_{i-1,j}$, resp.
$z_{i+1,j}$) if $w_i=\zeta$ and $i=1$ (resp. $w_i=\zeta$ and $i>1$, resp. $w_{i+1}=\zeta^{-1}$)
and $X_\zeta$ sends $z_{i,j}$ to 0 otherwise. Note that we set 
$z_{0,m+1}=0=z_{1,m+1}$ and $z_{n,j}=z_{0,j}$ for all $j$.
\end{dfn}


\end{document}